\documentclass[11pt]{amsart}
\usepackage{amsmath,amssymb,hyperref}
\usepackage{amsfonts}
\usepackage{enumerate}
\usepackage{amsmath, amsthm}
\usepackage{amscd}
\input xy
\xyoption{all}
\hypersetup{pdfpagemode=FullScreen, colorlinks=true}

\newtheorem{thm}{Theorem}[section]

\newtheorem{prop}[thm]{Proposition}
\newtheorem{lemma}[thm]{Lemma}

\newtheorem{cor}[thm]{Corollary}
\newtheorem{defi}[thm]{Definition}

\newtheorem{rem}[thm]{Remark}

\begin{document}
\title[Minimal volume of uniform visibility manifolds]
{Minimal volume of complete uniform visibility manifolds with finite volume}

\author{Sungwoon Kim}
\address{School of Mathematics,
KIAS, Heogiro 85, Dongdaemun-gu, Seoul, 130-722, Republic of Korea}
\email{sungwoon@kias.re.kr}


\begin{abstract}
We show that complete uniform visibility manifolds of finite volume with sectional curvature $-1 \leq K \leq 0$ have positive simplicial volumes. This implies that their minimal volumes are non-zero.
\end{abstract}

\maketitle

\section{Introduction}
The minimal volume of a smooth manifold $M$ is defined as the lower bound of the total volumes of all complete Riemannian metrics on $M$ whose sectional curvatures are bounded in absolute terms by one. Gromov \cite{Gr82} introduces the notion of the minimal volume and proves that the minimal volume is bounded from below by the simplicial volume, which is a type of topological invariant.
In the same paper, the question was naturally raised as to which manifolds have non-zero simplicial volumes.
Gromov conjectures that non-positively curved closed manifolds with negative
Ricci curvature have positive simplicial volumes.

First, it is verified by Gromov \cite{Gr82} and Thurston \cite{Th78} that complete Riemannian manifolds of finite volume with pinched negative sectional curvature have positive simplicial volumes.
Subsequently, research has focused on the simplicial volume of the locally symmetric spaces of non-compact type in an effort to explore the simplicial volume of Riemannian manifolds with non-positive sectional curvature.
It is proved by Lafont and Schmidt \cite{LS06} that
the simplicial volume of closed locally symmetric spaces of non-compact type is positive.
Also, closed visibility manifolds with non-positive sectional curvature, another type of manifold in the category of non-positively curved manifolds, have positive simplicial volumes.

Unlike the closed manifolds discussed above, the simplicial volume of non-compact Riemannian manifolds of finite volume is somewhat odd. For locally symmetric spaces of non-compact type, L\"{o}h and Sauer \cite{LS09-2} show that
the simplicial volume of locally symmetric spaces with $\mathbb{Q}$-rank of at least $3$ vanishes. On the other hand, it is verified that $\mathbb{Q}$-rank $1$ locally symmetric spaces covered by the product of $\mathbb{R}$-rank $1$ symmetric spaces have positive simplicial volumes \cite{KK11}, \cite{LS09-1}. The other $\mathbb{Q}$-rank $2$ case remains open.

In the case of non-compact visibility manifolds of finite volume, little is known about their simplicial volume and  minimal volume. The aim of this paper is to verify the positivity of the simplicial volume of non-compact uniform visibility manifolds with finite volume.

\begin{thm}\label{mainthm}
Let $M$ be a complete uniform visibility manifold of finite volume with sectional curvature $-1\leq K_M \leq 0$.
Then, the simplicial volume of $M$ is strictly positive.
\end{thm}

Let $\text{dim}M=n$.
The curvature condition $-1\leq K_M \leq 0$ gives the lower bound on the Ricci curvature of $M$, that is,  $\text{Ricci}_M \geq -(n-1)$. This guarantees the estimate of the minimal volume of $M$ given by Gromov \cite{Gr82}, as follows :
$$\| M \| \leq (n-1)^n  n! \cdot \text{Minvol}(M)$$
Hence, we obtain the following corollary immediately :

\begin{cor}
The minimal volume of complete uniform visibility manifolds of finite volume with sectional curvature
 $-1\leq K \leq 0$ is positive.
\end{cor}

Visibility manifolds are introduced by Eberlein and O'Neill \cite{EO73} as a generalization of strictly negative sectional curvature.
Eberlein \cite{Eb80} shows that
if $M$ is a complete uniform visibility manifold of finite volume with sectional curvature $\-1 \leq K_M \leq 0$,
then $M$ is tame; i.e., $M$ is the interior of some compact manifold with boundary.
Visibility manifolds are closely related to Gromov-hyperbolic spaces. Indeed, it turns out that the notion of uniform visibility is equivalent to the notion of Gromov-hyperbolicity. Recent works on relatively hyperbolic groups allow us to explore the simplicial volume of non-compact uniform visibility manifolds.

\section{Visibility manifold and hyperbolic space}

The notion of visibility can be generalized to $\text{CAT}(0)$-space. Eberlein and O'Neill \cite{EO73} first introduce the notion of visibility for Hadamard manifolds. Here, we recall the notion of visibility for $\text{CAT}(0)$-space in \cite{Br95}. Let $X$ be a $\text{CAT}(0)$-space. For $x,y,p \in X$, let $[x,y]$ denote the unique geodesic segment from $x$ to $y$ in $X$ and $x\widehat{p}y$ be the angle between $[p,x]$ and $[p,y]$ at $p$.

\begin{defi}
A $\text{CAT}(0)$-space $X$ is said to be locally visible if for every $p\in X$ and $\epsilon >0$, there exists $R(p,\epsilon) \geq 0$ such that if a geodesic segment $[x,y]$ lies entirely outside the ball of the radius $R(p,\epsilon)$ about p, then $x\widehat{p}y < \epsilon$. Moreover, $X$ is said to be uniformly visible if the constant $R(p,\epsilon)$ can be chosen such that it is independent of $p \in X$.
\end{defi}

A Riemannian manifold $M$ is said to be a visibility manifold if its universal cover is locally visible. Also, $M$ is said to be a uniform visibility manifold if its universal cover is uniformly visible.
It is well known that complete, simply connected Riemannian manifolds with strictly negative sectional curvature are uniformly visible.

\begin{thm}[Eberlein, \cite{Eb80}]\label{eberlein}
Let $X$ be a visibility manifold satisfying the curvature condition $-b\leq K \leq 0$. If $\Gamma$ is any non-uniform lattice in $X$, then $M=\Gamma\backslash X$ has only finitely many ends, and each end is a parabolic, Riemannian collared end. In particular, $\Gamma$ is finitely generated.
\end{thm}

Theorem \ref{eberlein} specifies that a non-compact, complete visibility manifold $M$ of finite volume with sectional curvature $-1\leq K_M \leq 0$ is tame. At this point, we recall the notion of Gromov-hyperbolic space.

\begin{defi}
Given $\delta>0$, a geodesic metric space $X$ is said to be $\delta$-hyperbolic if for every geodesic triangle $\Delta \subset X$, each edge of $\Delta$ is contained in the $\delta$-neighborhood of the union of the other two sides. $X$ is said to be hyperbolic if it is $\delta$-hyperbolic for some $\delta>0$.
\end{defi}

Origin ideas of the notions of visibility and hyperbolicity come from strictly negative sectional curvature.
It turns out that the notion of uniformly visibility is equivalent to the notion of hyperbolicity under the assumption of non-positive sectional curvature.

\begin{prop}[Bridson, \cite{Br95}]\label{propBrid}
Let $X$ be a $\text{CAT}(0)$-space.
\begin{itemize}
\item[(1)] $X$ is hyperbolic if and only if $X$ is uniformly visible.
\item[(2)] If $X$ is cocompact and locally visible, then it is uniformly visible (and hence hyperbolic).
\end{itemize}
\end{prop}

Proposition \ref{propBrid} clearly indicates that the fundamental group of closed visibility manifolds is a hyperbolic group. This induces that the simplicial volume of closed visibility manifolds is positive.

\section{Complete uniform visibility manifolds}

The notion of a relatively hyperbolic group was formulated by Gromov \cite{Gr87}.
Indeed, fundamental groups of non-compact, complete, finite volume Riemannian manifolds with pinched
negative sectional curvature are the motivating examples for formulating relatively hyperbolic groups.
Bowditch \cite{Bo00} gives two equivalent definitions of relatively hyperbolic groups, which are equivalent to the definition given in \cite{Gr87}. Here, we recall one of them.

\begin{defi}\label{defrelhyp}
Let $\Gamma$ be a group and $\mathcal{P}$ be a set of infinite subgroups. Then, $\Gamma$ is \textit{hyperbolic relative to} $\mathcal{P}$ if $\Gamma$ admits a properly discontinuous isometric action on a path-metric space $X$ with the following properties :
\begin{itemize}
\item[(1)] $X$ is proper and hyperbolic.
\item[(2)] Every point of the boundary of $X$ is either a conical limit point or a bounded parabolic point.
\item[(3)] The elements of $\mathcal{P}$ are precisely the maximal parabolic subgroups of $\Gamma$.
\item[(4)] Every element of $\mathcal{P}$ is finitely generated.
\end{itemize}
\end{defi}

Let $M$ be a non-compact, complete, finite volume Riemannian manifold with pinched negative sectional curvature. Let $\Gamma$ be the fundamental group of $M$, and let $\mathcal{P}$ be the set of
all maximal parabolic subgroups of $\Gamma$. In such a case, Farb \cite{Fa94} shows that $\Gamma$ is hyperbolic relative to $\mathcal{P}$. As the fundamental group of a closed visibility manifold is hyperbolic,
we observe that the fundamental group of non-compact, complete uniform visibility manifolds of finite volume with sectional curvature $-1\leq K_M \leq 0$ is hyperbolic relative to the set of all
maximal parabolic subgroups, as follows :

\begin{prop}\label{relhyp}
Let $M$ be a complete uniform visibility manifold of finite volume with sectional curvature $-1\leq K_M \leq 0$.
Then, $\Gamma$ is hyperbolic relative to
$\mathcal{P}$, where $\Gamma$ is the fundamental group of $M$ and $\mathcal{P}$ is the set of all maximal parabolic subgroups of $\Gamma$.
\end{prop}

\begin{proof}

Let $X$ be the universal cover of $M$. Because $M$ is a uniform visibility manifold, $X$ is uniformly visible and hence hyperbolic. If $M$ is closed, the Caley graph of $\Gamma$ is quasi-isometric to $X$.
Thus, $\Gamma$ is hyperbolic relative to $\mathcal{P}=\emptyset$, that is, hyperbolic.

At this stage, we suppose that $M$ is not closed. According to Theorem \ref{eberlein}, $M$ has only finitely many ends with each end being a parabolic, Riemannian collared end. More precisely, there exists a neighborhood $U_E$ of $E$, a compact $C^2$ codimension $1$ submanifold $N_E$ of $M$ for each end $E$ of $M$ and a $C^1$ diffeomorphism $F : N_E \times (0,\infty) \rightarrow U_E$ such that the curves $t \rightarrow F(n,t)$, $n\in N_E$, are unit speed distance minimizing geodesics of $M$ that intersect each hypersurface $F(N_E \times \{s\})$ orthogonally. Indeed, $N_E$ is the projection of a precisely invariant horosphere in $X$ at a point $p \in \partial X$ fixed by a maximal parabolic subgroup of $\Gamma$, and $U_E$ is the projection of the corresponding open horoball in $X$. Hence, $M$ is the interior of a compact manifold with boundary.

The tameness of $M$ implies that $\Gamma$ is finitely generated and that this is also true for each maximal parabolic subgroup of $\Gamma$. Clearly, $\Gamma$ acts properly discontinuously on $X$. Furthermore, it is clear that conditions $(1),(3)$ and $(4)$ are satisfied. Now, we only need to show condition $(2)$; that is, every point of the boundary of $X$ is either a conical limit point or a bounded parabolic point.

Let $\partial X$ denote the boundary of $X$ and $p \in \partial X$ be a parabolic point associated with a maximal parabolic subgroup $P$ of $\Gamma$.
Then, $P\backslash(\partial X - \{p\})$ is homeomorphic to $N_E$, as above, for some end $E$ of $M$. Because $N_E$ is compact, $p$ is a bounded parabolic point according to this definition. Thus, every parabolic point is a bounded parabolic point.

Let $\Pi$ be the set of all bounded parabolic points in $\partial X$ with respect to $\Gamma$.
It is clear that $\Pi$ is $\Gamma$-invariant. Moreover, $\Pi/\Gamma$ is finite because $M$ has finitely many ends.
According to \cite[Proposition 6.11]{Bo00}, there exists an invariant system $\mathcal{B}$ of horoballs, that is, a collection, $(B(p))_{p\in \Pi}$, indexed by $\Pi$, such that $B(p)$ is a horoball about $p$ and $B(\gamma p)=\gamma B(p)$ for all $\gamma \in \Gamma$ and all $p\in \Pi$. In this case, we have a closed $\Gamma$-invariant subset $$Y(\mathcal{B})=X-\bigcup_{p\in \Pi} \text{int}B(p).$$

The closed subset, $\Gamma\backslash Y(\mathcal{B})$ is a compact manifold with boundary $N_E$. According to \cite[Proposition 6.14]{Bo00}, it can be concluded that every point of $\partial X / \Pi$ is a conical limit point. We refer the reader to \cite[Section 6]{Bo00} for a more detailed explanation of this. Hence, every point of $\partial X$ is either a conical limit point or a bounded parabolic point, which implies that $\Gamma$ is hyperbolic relative to $\mathcal{P}$.
\end{proof}

If $-b\leq K_M \leq -a <0$, every maximal parabolic subgroup of $\Gamma$ is virtually nilpotent according to the Margulis lemma. This does not hold for general uniform visibility manifolds with non-positive sectional curvature.
However, we observe that every maximal parabolic subgroup of $\Gamma$ is virtually nilpotent for a non-uniform lattice $\Gamma$ in a uniformly visible space $X$.

\begin{prop}\label{nilpotent}
Let $\Gamma$ be the fundamental group of a noncompact, complete uniform visibility manifold $M$ of finite volume with sectional curvature  $-1\leq K_M\leq 0$. Then, every maximal parabolic subgroup of $\Gamma$ is virtually nilpotent.
\end{prop}

\begin{proof}
Let $X$ be the universal cover of $M$.
Let $\mathcal{P}$ be the set of all maximal parabolic subgroups of $\Gamma$. Then, $\Gamma$ is hyperbolic relative to $\mathcal{P}$, as shown in Proposition \ref{relhyp}.
Dahmani and Yaman \cite{DY05} prove that every element of $\mathcal{P}$ is virtually nilpotent if and only if $X$ is geometrically bounded. Recall that a space $X$ is geometrically bounded if there exists a function $f : \mathbb{R}_+ \rightarrow \mathbb{R}_+$ such that
for all $R > 0$, every ball of radius $R$ can be covered by $f(R)$ balls of radius $1$ and every ball of radius $1$
can be covered by $f(R)$ balls of radius $1/R$.

Now, we claim that $X$ is geometrically bounded due to the sectional curvature condition of $-1\leq K_X \leq 0$.
Let $\text{dim}X=n$. Let $B^{\kappa}(R)$ be the geodesic ball of radius $R$ in the complete, simply connected Riemannian model space of constant curvature $\kappa$.
It follows from the comparison of the volumes of geodesic balls given by Bishop-G\"{u}nther-Cheeger-Gromov that for every $p\in X$, we have the inequality
$$ vol(B^0(R)) \leq vol(B_p(R)) \leq vol(B^{-1}(R)),$$
where $B_p(R)$ is the geodesic ball of radius $R$ centered at $p$.

Let $\mathcal{V}$ be a finite set of points in $B_p(R)$ such that
\begin{itemize}
\item[(1)] any point of $\mathcal{V}$ lies at distance at least $1/2$ from the boundary of $B_p(R)$,
\item[(2)] any two points of $\mathcal{V}$ lie at a distance at least $1$ from each other, and
\item[(3)] for all $x\in B_p(R)$, there exists $y \in \mathcal{V}$ such that the distance from $x$ to $y$ is less than $1$.
\end{itemize}
A set $\mathcal{V}$ is obtained by successively marking points in $X$ at pairwise
distances $\geq 1$ until there is no more room for such points.
Then, it becomes clear that
$\{B_x(1) \}_{x\in \mathcal{V}}$ is a covering of $B_p(R)$ and that $\{B_x(1/2) \}_{x\in \mathcal{V}}$ is the set of pairwise disjoint balls totally contained in $B_p(R)$.
Hence, we have the following inequality
$$ |\mathcal{V}|\cdot vol (B^0(1/2))\leq \sum_{x\in \mathcal{V}} vol(B_x(1/2 )) \leq
 vol(B_p(R)) \leq vol(B^{-1}(R)).$$

Let $f_1(R)$ be the nearest integer to $vol(B^{-1}(R))/ vol(B^0(1/2))$. Then, every ball of radius $R$ in $X$ can be covered by $f_1(R)$ balls of radius $1$. In a similar argument, every ball of radius $1$ in $X$ can be covered by $f_2(R)$ balls of radius $1/R$ where $f_2(R)$ is the nearest integer to $vol(B^{-1}(1))/ vol(B^0(1/2R))$. Define $f : \mathbb{R}_+ \rightarrow \mathbb{R}_+$ by $$f(R) = \max \{ f_1(R),f_2(R) \}.$$
Then, we can conclude that $X$ is geometrically bounded. This completes the proof.
\end{proof}

\begin{lemma}\label{finiteness}
Let $M$ be a complete uniform visibility manifold of finite volume with sectional curvature $-1 \leq K_M \leq 0$. Then, the simplicial volume of $M$ is finite.
\end{lemma}
\begin{proof}
If $M$ is closed, it is clear. Suppose that $M$ is non-compact. Then, $M$ is the interior of a compact manifold $V$ with boundary. According to Proposition \ref{nilpotent}, the fundamental group of $\partial V$ is virtually nilpotent and hence amenable. Note that the bounded cohomology of an amenable group
vanishes. Due to the duality between the $\ell^1$-homology and the bounded cohomology in \cite[Corollary 5.1]{Lo08}, the $\ell^1$-homology of $\partial V$ also vanishes. This means that
the fundamental class of $\partial V$ vanishes in the $\ell^1$-homology of $\partial V$.
According to the finiteness of the criterion in \cite[Theorem 6.4]{Lo08}, the simplicial volume of $M$ is finite.
\end{proof}

\begin{rem}
The sectional curvature condition $-1 \leq K \leq 0$ in Lemma \ref{finiteness} is essential.
Here is a counterexample.
Let $M$ be a closed hyperbolic $n$-manifold and $N$ be a totally geodesic, embedded, codimension $1$ closed submanifold of $M$. Delete the $\epsilon$-tubular neighborhood $U$ of $N$ for a sufficiently small $\epsilon >0$. Let $W$ be a component of $M- U$. Then, $W$ admits a complete metric of finite volume with  sectional curvature $K_W \leq -1$ \cite{Ph11}.

It follows from the sectional curvature condition $K_W \leq -1$ that $W$ is a complete uniform visibility manifold of finite volume. Furthermore, $W$ is tame.
However, a component of $\overline{W}$ is homeomorphic to the closed hyperbolic manifold $N$.
As the simplicial volume of $N$ is strictly positive, it is impossible for the fundamental class of $N$ to vanish in the $\ell^1$-homology of $N$. Hence, the simplicial volume of $W$ is not finite. We refer the reader to \cite{Ph11} for more details about the construction of $W$.
\end{rem}

\section{Simplicial volume and minimal volume}

We now prove that complete uniform visibility manifolds of finite volume with sectional curvature $-1\leq K \leq 0$ have positive simplicial volumes and are therefore the minimal volume.

\subsection{Relative hyperbolicity and bounded cohomology}
First, we need to look at the definition of relative hyperbolicity as given by Mineyev and Yaman in order to use their result of the bounded cohomology of relatively hyperbolic groups. In fact,
they slightly generalize Bowditch's combinational formulation of relative hyperbolicity, as follows :

\begin{defi}\label{mineyev}
Let $\Gamma$ be a group and $\mathcal{P}=\{ \Gamma_i \text{ }|\text{ }i\in I\}$ be a family of its subgroups. $\Gamma$ is called relatively hyperbolic with respect to $\mathcal{P}$ if there exists a graph $\mathcal{K}$ on which $\Gamma$ acts such that the following conditions are satisfied.
\begin{itemize}
\item $\Gamma$ is finitely generated.
\item $I$ is finite and each $\Gamma_i$ is finitely generated.
\item $\mathcal{K}$ is fine and has thin triangles.
\item There are finitely many orbits of edges and each edge stabilizer is finite.
\item There exists a $\Gamma$-invariant subset $\mathcal{V}'$ such that $\mathcal{V}_\infty \subset \mathcal{V}' \subset \mathcal{V}$ and the stabilizers of vertices in $\mathcal{V}'$ are precisely $\Gamma_i$ and their conjugates.
\end{itemize}
\end{defi}

Definition \ref{mineyev} allows the elements of $\mathcal{P}$ to be finite as well as infinite. In contrast,
Definition \ref{defrelhyp} only allows the elements of $\mathcal{P}$ to be infinite. Note that the family $\mathcal{P}$ of subgroups is the set of all maximal parabolic subgroups in Definition \ref{defrelhyp}, but in Definition \ref{mineyev}, $\mathcal{P}$ should be thought of as the set of conjugacy classes of maximal parabolic subgroups. Clearly, Definition \ref{defrelhyp} implies Definition \ref{mineyev}.

\begin{lemma}\label{comparison}
Let $V$ be a compact manifold with boundary whose interior is homeomorphic to a complete uniform visibility manifold of finite volume with non-positive sectional curvature bounded from below.
Then, $$H_b^k(V,\partial V) \rightarrow H^k(V,\partial V),$$
is surjective for all $k \geq 2$.
\end{lemma}

\begin{proof}
Let $M$ be a complete uniform visibility manifold of finite volume with sectional curvature bounded from below that is homeomorphic to the interior of $V$. We can assume sectional curvature $-1 \leq K_M \leq 0$ by scaling the metric on $M$.
Let $\Gamma$ be the fundamental group of $M$ and let $\mathcal{P}$ be the set of all maximal parabolic subgroups of $\Gamma$.

As shown in Proposition \ref{relhyp}, $\Gamma$ is hyperbolic relative to $\mathcal{P}$. There are finitely many conjugacy classes of maximal parabolic subgroups of $\Gamma$, as $M$ has finitely many parabolic, Riemannian collared ends.
Let $[P_1],\ldots,[P_l]$ denote the conjugacy classes of $\mathcal{P}$, where $P_1,\ldots,P_l$ are the maximal parabolic subgroups in $\mathcal{P}$. We set $\mathcal{P}=\{P_i\text{ }|\text{ }i=1,\ldots,l\}$ according to the abuse of notation.
Then, the pair $(\Gamma,\mathcal{P})$ is also hyperbolic in the sense of Mineyev and Yaman.
It follows from \cite[Theorem 59]{MY07} that the relative comparison map
$$c : H^k_b(\Gamma,\mathcal{P}) \rightarrow H^k(\Gamma,\mathcal{P})$$
is surjective for all $k\geq 2$

Each end of $M$ is associated with a conjugacy class of the maximal parabolic subgroups of $\Gamma$. Let $E_i$ denote the end of $M$ associated with $P_i$ for each $i=1,\ldots,l$. Then, there exists an open horoball $H_i$ in $X$ such that $U_i=P_i\backslash H_i$ is a neighborhood of $E_i$ which is diffeomorphic to $N_i \times (0,\infty)$ for each $i=1,\ldots,l$, where $N_i$ is the projection of a horosphere in $H_i$ fixed by $P_i$. Moreover, $U_i$ are pairwise disjoint subspaces of $M$.

Note that $M$ is a classifying space of $\Gamma$ because $X$ is contractible. Also, every horoball in $X$ is contractible; hence, $U_i$ is a classifying space of $P_i$. Let $U=\bigcup_{i=1}^l U_i$.
Then, $(M,U)$ is a classifying space for $(\Gamma,P)$ in the sense of \cite[Section 9.1]{MY07}. This implies that the comparison map $H^k_b(M,U)\rightarrow H^k(M,U)$ is identical to the map $H^k_b(\Gamma,\mathcal{P}) \rightarrow H^k(\Gamma,\mathcal{P})$ for all $k\geq 0$.
Given that $U$ is the collared neighborhood of $\partial V$ in $V$, it is clear that $(M,U)$ and $(V,\partial V)$ are homotopy equivalent. Finally, we can conclude that the comparison map $H^k_b(V,\partial V)\rightarrow H^k(V, \partial V)$ is identical to the map $H^k_b(\Gamma,\mathcal{P}) \rightarrow H^k(\Gamma,\mathcal{P})$ for all $k\geq 0$ and is hence surjective for all $k\geq 2$.
\end{proof}

\subsection{Simplicial volume}
Let $X$ be any topological space and $Y$ be the subset of $X$. Then, the $\ell^1$-norm in the real relative singular chain complex $C_*(X,Y)$ is defined by $\|c\|_1 =\sum |a_i|$ for $c=\sum a_i \sigma_i$ in $C_*(X,Y)$. This $\ell^1$-norm gives rise to a seminorm on the homology $H_*(X,Y)$, as follows :
$$\| \alpha \|_1 = \inf \|z\|_1,$$
where $z$ runs over all singular cycles representing $\alpha \in H_*(X,Y)$.

For a compact manifold $M$, the simplicial volume $\|M ,\partial M \|$ of $M$ is defined as the seminorm of the relative fundamental class $[M,\partial M]$ of $M$. If $\partial M = \emptyset $, the simplicial volume of $M$ is denoted by $\| M \|$.

If $M$ is an $n$-dimensional non-compact manifold, its fundamental class is well defined in the locally finite homology $H^\text{lf}_*(M)$ of $M$ with trivial coefficient. The locally finite homology $H^\text{lf}_*(M)$ of $M$ is defined as the homology of the locally finite chain complex $C_*^\text{lf}(M)$.
More precisely, let $S_k(M)$ be the set of singular $k$-simplices of $M$ and let $S^\text{lf}_k(M)$ denote the set of all locally finite subsets of $S_k(M)$; that is, if $A \in S^\text{lf}_k(M)$, any compact subset of $M$ intersects the image of only finitely many elements of $A$. The locally finite chain complex $C_*^\text{lf}(M)$ is then defined by
$$C_*^\text{lf}(M)= \Big\{\sum_{\sigma \in A} a_\sigma \cdot \sigma \text{ }\Big| \text{ }A \in S^\text{lf}_*(X) \text{ and }a_\sigma \in \mathbb{R} \Big\}.$$

As the $\ell^1$-seminorm on $H_*(M)$ is induced, the $\ell^1$-seminorm on $H^\text{lf}_*(M)$ is
induced from the $\ell^1$-norm on the locally finite chain complex $C^\text{lf}_*(M)$ with respect to the basis given by all singular simplices. Because $H^\text{lf}_n(M,\mathbb{Z}) \cong \mathbb{Z}$, the fundamental class of $M$ is well defined in $H^\text{lf}_n(M) \cong \mathbb{R}$.
The simplicial volume of $M$ is defined as the $\ell^1$-seminorm of the locally finite fundamental class of $M$.
In particular, if $M$ is the interior of a compact manifold $V$, then we have the inequality
$$ \| V,\partial V \| \leq \| M \|.$$
This can be shown by the cohomological definition of the simplicial volume.
For more details, the reader can refer to \cite{Gr82}.

\begin{thm}
Let $M$ be a complete uniform visibility manifold of finite volume with sectional curvature $-1\leq K_M \leq 0$.
Then the simplicial volume of $M$ is strictly positive.
\end{thm}

\begin{proof}
Let $V$ be a compact manifold with boundary whose interior is homeomorphic to $M$.
Then, the relative simplicial volume of $V$ can be computed in terms of the bounded cohomology of $(V,\partial V)$ as follows :
$$\| V,\partial V \| = \sup \bigg\{ \frac{1}{\|\omega\|_\infty}\text{ }\bigg|\text{ }\omega\in H^n_b(V,\partial V)\text{ and } \langle \omega, [V,\partial V] \rangle =1 \bigg\},$$
where $[V,\partial V]$ is the relative fundamental class of $V$.
Here, $\sup \emptyset = 0$.

The existence of a bounded cohomology class $\omega$ satisfying $\langle \omega, [V,\partial V] \rangle =1$ implies the positivity of the simplicial volume $\|V,\partial V \|$. It is a standard fact that there exists a dual cohomology class $[V,\partial V]^*$ in $H^n(V,\partial V)$ satisfying $ \langle [V,\partial V]^*,  [V,\partial V] \rangle =1$. According to Lemma \ref{comparison}, there exists a bounded cohomology class $[V,\partial V]^*_b$ in $H^n_b(V,\partial V)$ representing $[V,\partial V]^*$. One can easily check that $ \langle [V,\partial V]^*_b,  [V,\partial V] \rangle =1 .$
Therefore, the simplicial volume $\|V,\partial V\|$ is positive.
From the inequality
$$0<\| V,\partial V\| \leq \|M\|,$$
it follows that $\| M \|$ is strictly positive.
\end{proof}

For an $n$-dimensional smooth manifold $M$ with $\text{Ricci}_M \geq -(n-1)$, Gromov proves that
$$\| M\| \leq (n-1)^n n! \cdot \text{Minvol}(M).$$

Observe that a bound from below for sectional curvature, $K_M \geq -1$, implies $\text{Ricci}_M \geq -(n-1)$.
Hence, the following corollary is obtained immediately :

\begin{cor}
The minimal volume of complete uniform visibility manifolds of finite volume with sectional curvature
 $-1\leq K_M \leq 0$ is positive.
\end{cor}

\end{document}